\documentclass{amsart}
\usepackage{amsmath}

\usepackage{latexsym, amssymb, amsthm}

\title{Learning Theory in the Arithmetic Hierarchy}

\keywords{Inductive Inference, Learning Theory, Arithmetic Hierarchy}
\subjclass[2010]{Primary 03D80, 68Q32}

\author{Achilles A. Beros}

\address{Department of Mathematics\\
University of Wisconsin - Madison\\
Madison, WI 53706}

\email{aberos@math.wisc.edu}

\newtheorem{Theorem}{Theorem}[section]

\newtheorem{Lemma}[Theorem]{Lemma}
\newtheorem{Corollary}[Theorem]{Corollary}

\theoremstyle{definition}
\newtheorem{Definition}[Theorem]{Definition}

\newtheorem{example}[Theorem]{Example}

\newcommand{\upto} {\negthickspace \upharpoonright \negmedspace}

\begin{document}

\begin{abstract}
We consider the arithmetic complexity of index sets of uniformly computably enumerable families learnable under different learning criteria.  We determine the exact complexity of these sets for the standard notions of finite learning, learning in the limit, behaviorally correct learning and anomalous learning in the limit.  In proving the $\Sigma_5^0$-completeness result for behaviorally correct learning we prove a result of independent interest;  if a uniformly computably enumerable family is not learnable, then for any computable learner there is a $\Delta_2^0$ enumeration witnessing failure.
\end{abstract}

\maketitle

Algorithmic learning theory examines the process by which members of a class are identified from a finite amount of information.  The classes to be learned are either classes of functions or classes of computably enumerable ($c.e.$) sets.  In this paper we address the learning of $c.e.$ sets.\par

Since learning is not a mathematical concept, it is not endowed with an unambiguous definition.  Like the concept of computability, learning has an intuitive meaning, but lends itself to a number of different formalizations.  In learning theory, we consider effective formalizations and call them models of learning.  A model is principally defined by two factors: the type of information and the criterion for success.  The information is read by a learning machine from an enumeration of the set to be learned.  As the machine must be computable, it cannot consider the entirety of an enumeration and will only take a finite initial segment as input.  On such an input, the learning machine outputs a natural number, interpreted as a $\Sigma_1^0$-code describing the content of the set being enumerated.  We call such outputs \textit{hypotheses}.  As the machine reads longer initial segments of the enumeration, it outputs a sequence of hypotheses that we will call the \textit{hypothesis stream}.  The condition on when the hypothesis stream represents successful learning is the criterion for success and may depend on the accuracy, consistency or frequency of correct information in the hypothesis stream.  Additional limitations, such as bounds on the computational resources of the learning machine, are often considered.\par

The first model of learning is due to Gold ~\cite{ex}.  According to his model, now commonly referred to as TxtEx-learning, a machine is deemed to have successfully identified an enumeration if, on cofinitely many initial segments, the machine outputs the same hypothesis and it is correct.  A machine is said to have learned a set if it identifies every enumeration of the set, and has a learned a family if it learns every member of the family.  Three other standard notions are TxtFin-learning, TxtBC-learning and TxtEx$^*$-learning which differ from TxtEx-learning in what constitutes successful identification of an enumeration.  In the case of TxtFin-learning, while a machine is permitted to abstain from making a hypothesis for a finite amount of time, it must eventually output a hypothesis and the first such hypothesis must be correct.  A machine is said to have TxtBC-identified an enumeration if all but finitely many of the hypotheses in the hypothesis stream are correct.  In contrast to TxtEx-learning, they need not be the same.  Last, TxtEx$^*$-learning differs from TxtEx-learning in that the unique hypothesis appearing infinitely many times in the hypothesis stream need only code a set having finite symmetric difference with the content of the given enumeration.\par

With this paper, we introduce a new line of inquiry to the field of learning theory.  We examine the complexity of determining whether a family is learnable given a code for an effective presentation of the family, thereby establishing a measure of the complexity of the learning process.\par

We prove that decision problems for learning under the standard notions of TxtFin-learning and TxtEx-learning are $\Sigma_3^0$-complete and $\Sigma_4^0$-complete, respectively, and that those for TxtBC-learning and TxtEx$^*$-learning are both $\Sigma_5^0$-complete, when certain natural limitations are placed on the complexity of the families considered.  In proving the $\Sigma_5^0$-completeness of TxtBC-learning, we obtain a TxtBC-learning analog of a theorem of Blum and Blum~\cite{bb}.  Blum and Blum's work demonstrated that a set is TxtEx-learnable if it is TxtEx-learnable from computable enumerations.  We show that, provided the family under consideration is uniformly computably enumerable ($u.c.e.$), one need only consider $\Delta_2^0$ enumerations to decide if a family is TxtBC-learnable.\par

We preface the completeness results with a brief introduction to some of the concepts of learning theory and notation from computability theory.  For a more in depth treatment, we refer the reader to Osherson \textit{et al.} ~\cite{stl} and Soare~\cite{Soare}.\par

\section{Preliminaries}

Unless noted otherwise, all families in this paper are $u.c.e$.  We regard a set, $A$, as coding a family, $\mathcal F$, where the $i^{th}$ member (or \textit{column}) of $\mathcal F$ is $\{ x : \langle i,x \rangle \in A \}$ for a computable pairing function $\langle x,y \rangle$.  Given natural numbers, $e$ and $s$, $W_{e,s}$ denotes the result of computing the set coded by $e$ up to $s$ stages using a standard numbering of the $c.e.$ sets.  Finite strings of natural numbers are represented by lowercase Greek letters.  Enumerations, called texts in learning theory, are either treated as infinite strings or as functions on the natural numbers.  Initial segments of enumerations feature throughout this paper and are either denoted by lowercase Greek letters, as mentioned above, or by initial segments of functions: i.e. $T[n]$ in learning theory notation, or $f \upto n$ in standard logic notation.  The $n^{th}$ element of an enumeration is denoted $T(n)$ or $f(n)$, as is appropriate.  To switch from ordered lists to unordered sets, we say that content$(\sigma) = \{x\in \mathbb N : (\exists n) ( x = \sigma(n) )\}$.  For infinite enumerations, we extend the content notation to denote the set that is enumerated.  If $A$ and $B$ are sets of natural numbers and their symmetric difference, $A\triangle B$, is finite, then we write $A =^* B$.\par

Learning machines are denoted by $M$ or $N$, with subscripts or superscripts as needed to indicate parameters.  We consider an effective enumeration of all computable learning machines as having been fixed, whereby $M_n$ denotes the $n^{th}$ learner.\par

\begin{Definition}[\cite{fin}]
Fix a symbol, `?', as a placeholder to indicate that a hypothesis has not yet been made.  The definition of TxtFin-learning by a learner, $M$, is in four parts:
\begin{enumerate}
\item $M$ {\em TxtFin-identifies} an enumeration $f$ if and only if $(\exists n) (\forall n'<n) (M(f \upto n')=\mbox{?} \wedge M(f \upto n)\not = \mbox{?} \wedge W_{M(f \upharpoonright n)} = \mbox{content}(f))$.
\item $M$ {\em TxtFin-learns} a $c.e.$ set $A$ if and only if $M$ TxtFin-identifies every enumeration for $A$.
\item $M$ {\em TxtFin-learns} a family of $c.e.$ sets if and only if $M$ TxtFin-identifies every member of the family.
\item A family, $\mathcal F$, is {\em TxtFin-learnable} (denoted $\mathcal F \in \mbox{TxtFin}$) if and only if there is a machine, $M$, that TxtFin-learns $\mathcal F$.
\end{enumerate}
\end{Definition}

\begin{Definition}[\cite{ex}]
The definition of TxtEx-learning is analogous to that of TxtFin-learning.  TxtFin is everywhere replaced by TxtEx and the first clause is replaced by:
\begin{enumerate}
\item  $M$ {\em TxtEx-identifies} an enumeration $f$ if and only if $(\exists n) (\lim_{i \rightarrow \infty}M(f \upto i) = n \wedge W_n = \mbox{content}(f))$.
\end{enumerate}
\end{Definition}

\begin{Definition}[\cite{bc}]
The definition of TxtBC-learning is analogous to that of TxtFin-learning.  TxtFin is everywhere replaced by TxtBC and the first clause is replaced by:
\begin{enumerate}
\item  $M$ {\em TxtBC-identifies} an enumeration $f$ if and only if $(\exists n) (\forall i>n) (W_{M(f \upharpoonright i)}  = \mbox{content}(f))$.
\end{enumerate}
\end{Definition}

\begin{Definition}[\cite{ex*}]
The definition of TxtEx$^*$-learning is analogous to that of TxtFin-learning.  TxtFin is everywhere replaced by TxtEx$^*$ and the first clause is replaced by:
\begin{enumerate}
\item  $M$ {\em TxtEx$^*$-identifies} an enumeration $f$ if and only if $(\exists n) (\lim_{i \rightarrow \infty}M(f \upto i) = n \wedge W_n =^* \mbox{content}(f))$.
\end{enumerate}
\end{Definition}

Before turning to our own results, we present the following facts which will be needed in the subsequent sections.\par

We use the following theorem in the proof of Theorem \ref{exUp} and the corollary that follows in the proof of Theorem \ref{ex*Up}.

\begin{Theorem}[Blum and Blum ~\cite{bb}]\label{blum}
If a family is TxtEx-learned from computable enumerations by a computable machine $M$, then it is TxtEx-learned from arbitrary enumerations by a computable machine $\hat{M}$.
\end{Theorem}

\begin{Corollary}\label{blum*}
If $\mathcal F$ is TxtEx$^*$-learned from computable enumerations by a computable machine $M$, then it is TxtEx$^*$-learned from arbitrary enumerations by a computable machine $\hat{M}$.
\end{Corollary}

\begin{proof}
The proof is immediate.  Simply replace TxtEx-learning with TxtEx$^*$-learning throughout the proof of Theorem \ref{blum}.
\end{proof}

%
%

Next we have Angluin's Theorem.  The application of the theorem which follows is used in the proof of Theorem \ref{exlHard}. 

\begin{Theorem}[Angluin's Theorem ~\cite{angluin}]\label{ang} 
Let $\mathcal L = \{L_0, L_1, \ldots\}$ be a uniformly computable family.  $\mathcal L$ is $TxtEx$-learnable if and only if there is a $u.c.e.$ family of finite sets $\mathcal F = \{F_0, F_1,\ldots\}$ such that
\begin{enumerate}
 \item $F_i\subseteq L_i$ for all $i\in \mathbb N$
\item If $F_i\subseteq L_j\subseteq L_i$, then $L_i=L_j$
\end{enumerate}
\end{Theorem}

Although we do not present it here, examination of the proof of Angluin's Theorem shows that the existence of $\mathcal F$ does not require $\mathcal L$ to be uniformly computable, merely $u.c.e.$  Example \ref{bcNotEx} uses this observation to distinguish TxtBC-learning and TxtEx-learning.\par

\begin{example}\label{bcNotEx}
Let $H_x = \{x+n:n\leq |W_x|\}$, and $L_x = \{x+n:n\in \mathbb N\}$.  Define $\mathcal F = \{H_0, L_0, H_1, L_1, \ldots \}$.  We claim that $\mathcal F$ is TxtBC-learnable, but not TxtEx-learnable.\par  

$\mathcal F$ is clearly $u.c.e.$ and can be enumerated so that $H_e$ is the $(2e)^{th}$ column and $L_e$ is the $(2e+1)^{st}$ column of $\mathcal F$.  We must verify that $\mathcal F$ is TxtBC-learnable, but that no machine can TxtEx-learn $\mathcal F$.  Consider a machine that, on input $\sigma$, sets $x_0$ and $x_1$ to be the least element and greatest element, respectively, of content$(\sigma)$ and sets $y_1$ equal to the greatest element of $W_{x_0,|\sigma|}$.  If $y_1 > x_1 - x_0$, the machine outputs a code for $H_{x_0}$; if $x_1 - x_0 \geq y_1$, it outputs a code for $L_{x_0}$.  Thus, for enumerations of infinite intervals, the machine may vacillate between two different correct codes.  For finite intervals, eventually only one correct code will be output.  We conclude that $\mathcal F$ is TxtBC-learnable.\par

Now, we wish to show that $\mathcal F$ is not TxtEx-learnable.  To obtain a contradiction, assume that we have a machine, $M$, that TxtEx-learns $\mathcal F$.  As we observed at the outset, this means there is a $u.c.e.$ family $\{G_0, G_1, \ldots\}$ such that each $G_i$ is finite, $G_{2i} \subseteq H_i$, $G_{2i+1} \subseteq L_i$ and, if $G_i \subseteq A \subseteq B$ where $B$ is the $i^{th}$ set in $\mathcal F$ and $A \in \mathcal F$, then $A=B$.  Specifically, $G_{2i+1} \subseteq L_i$ and $G_{2i+1} \subseteq H_i$ if and only if  $H_i=L_i$ - exactly when $W_i$ is an infinite set.  Let $m$ be the maximum number in $G_{2i+1}$.  If card$(W_i) \geq m-i+1$, then $G_{2i+1} \subseteq H_i$ and $H_{i} = L_{i}$.  In other words, we can decide in the limit whether or not $W_i$ is infinite.  Since we cannot actually decide in the limit whether or not a number codes an infinite set, we have obtained the desired contradiction.
\end{example}  

We conclude this section with the following definition.

\begin{Definition}
Define the following four index sets of $\Sigma_1^0$ codes for $u.c.e.$ families learnable according to the given criterion.  
\begin{enumerate}
\item Define $FINL$ to be the index set for TxtFin-learning.
\item Define $EXL$ to be the index set for TxtEx-learning.
\item Define $BCL$ to be the index set for TxtBC-learning.
\item Define $EXL^*$ to be the index set for TxtEx$^*$-learning.
\end{enumerate}
\end{Definition}

\section{TxtFin-Learning}

We begin the presentation of our results by demonstrating that TxtFin-learning is $\Sigma_3^0$-complete.  This is accomplished in two steps.  With Theorem \ref{finUp} we place a $\Sigma_3^0$ upper bound on the complexity of FINL.  Next, Theorem \ref{finHard} reduces an arbitrary $\Sigma_3^0$ predicate to FINL.

\begin{Theorem}\label{finUp}
FINL has a $\Sigma^0_3$ description.
\end{Theorem}

\begin{proof}
Suppose $e$ codes a $u.c.e.$ family $\mathcal L = \{L_0, L_1, \ldots\}$.  We will show that $e \in FINL$ if and only if
$$(\exists k) (\forall i) \bigg( (\exists \sigma) \Big((\mbox{content}(\sigma)\subseteq L_i) \wedge(M_k(\sigma)\neq ?)\Big) \wedge \psi (k) \bigg)$$
where $\psi (k)$ denotes
$$(\forall \alpha,j) (\exists \tau \prec \alpha) \Big( (M_k(\tau) \neq ?)\vee(\mbox{content}(\alpha)\not\subseteq L_{j})
\vee(M_k(\alpha) = ?)\vee(W_{M_k(\alpha)} = L_j) \Big).$$

Observe that the formula mandates the existence of a learning machine, $M_k$, such that for every set in the family there is a string of elements from that set on which the learner outputs a hypothesis.  Furthermore, if it outputs a least hypothesis, in the sense that the only hypothesis it makes on proper initial segments of the given data is ?, then that hypothesis is correct.  We now build a new machine, $\hat{M}$, based on $M_k$, which TxtFin-learns the family $\mathcal L$.\par

For a string $\sigma$, define $A_{\sigma}=\{\tau: \mbox{content}(\tau)\subseteq \mbox{content}(\sigma) \wedge |\tau|\leq |\sigma|\}$.  Order $A_{\sigma}$ by $\sigma < \tau$ if either $|\sigma| < |\tau|$ or else $|\sigma|=|\tau|$ and $\sigma$ is below $\tau$ in the lexicographical order on $\mathbb N^{|\sigma|}$.  Define $\hat{M}(\sigma)$ to be $M_k(\tau)$ where $\tau$ is the least element of $A_{\sigma}$ on which $M_k$ outputs a hypothesis other than ?.   If no such $\tau$ exists, $\hat M(\sigma) = $?.  On no enumeration will the least hypothesis of $\hat{M}$ be incorrect since that would imply the existence of such an enumeration for $M_k$.  Fix an arbitrary enumeration, $f$, for $L_i\in \mathcal L$.  Let $\sigma$ be a string, with content$(\sigma) \subseteq L_i$, on which $M_k$ outputs a least (and hence correct) hypothesis.  Every element of content$(\sigma)$ appears in $f$, thus there is an $n$ such that $\sigma \in A_{f \upharpoonright n}$.  For some $m\leq n$, $\hat{M}(f \upto m)$ will be a least and correct hypothesis.\par

Since the given formula is $\Sigma^0_3$, we have produced a predicate with the desired properties.

\end{proof}

\begin{Theorem}\label{finHard}
FINL is $\Sigma^0_3$-hard.
\end{Theorem}

\begin{proof}
Consider a $\Sigma_3^0$ predicate $P(e) \leftrightarrow (\exists x) (\forall y) (\exists z) (R(e,x,y,z))$, where $R$ is a computable predicate.  We will reduce $P$ to FINL by means of a computable function such that the image of $e$ is a code for a $u.c.e.$ family that is TxtFin-learnable if $P(e)$ and not TxtFin-learnable if $\neg P(e)$.  We now fix $e$ and proceed with the construction of a family based on that particular $e$.  The family under construction is denoted $\mathcal G = \{ G_0, G_1, \ldots \}$.  While $\mathcal G$ depends on $e$, we omit the parameter for the sake of simplicity as we are only concerned with the fixed value of $e$ during the construction below.\par

Each $G \in \mathcal G$ will consist of ordered pairs and can thus be partitioned into columns $C(G,i) = \{ x : \langle i,x \rangle \in G \}$.  For convenience, we will index the columns starting with $-1$.  Let $C(i) = \{ \langle i,x \rangle : x \in \mathbb N \}$.  For the remainder of the construction we will adhere to the notation defined in the following list.

\begin{itemize}

\item Let $\langle x_s,y_s,z_s \rangle$ be a computable enumeration of all triples of natural numbers.

\item For $x \geq 1$, let $h_x^s$ be the number of stages, up to $s$, at which the largest $j$, such that $(\forall y \leq j) (\exists i \leq s) ( x_i = x \wedge y_i = y \wedge R(e,x-1,y,z_i))$, has increased.

\item If it exists, let $h_x = \lim_{s \rightarrow \infty} h_x^s$.

\item For $x\in \mathbb N$, an \textit{$x$-label} is a number in the $x^{th}$-column, $C(x)$, used to distinguish sets in $\mathcal G$.  Labels may be enumerated into any column of any $G\in \mathcal G$ except $C(G,-1)$ during the construction.  We say $G$ has an $x$-label $k$ when $k \in C(G,x)$.  Equivalently, when $\langle x,k \rangle \in G$.

\item Define $\mathcal S_x^k = \{ G \in \mathcal G : \mbox{$G$ has an $x$-label $k$} \}$.  The family depends on the stage, but we do not include any notation to indicate the stage as it will be clear from context.  When we wish to reference the $i^{th}$-member of $\mathcal S_x^k$, we will write $S_x^k(i)$.

\item Let $n_x^k = {\rm card}(\mathcal S_x^k)$.

\item Define a function, $p_x^k$, used to record numbers associated with each member of $\mathcal S_x^k$.  In particular, $p_x^k(i)$ will be a number withheld from $S_x^k(i)$.  Denote by $P_x^k$ the set $\{ p_x^k(0),\ldots,p_x^k(n_x^k) \}$.  At each stage, we will ensure that $P_x^k \setminus \{ p_x^k(i) \} \subseteq C(S_x^k(i),-1)$ and $p_x^k(i) \not\in C(S_x^k(i),-1)$.  At certain stages, the values of the $p_x^k$ will change.

\end{itemize}

Next, we describe the actions taken at a given stage of the construction.  The construction consists of using the predicate, $P$, to resolve two opposing forces.  One is the attempt to label all sets in a unique way, and the other is to create an infinite family that mirrors the structure of $\{ \mathbb N \setminus \{x\} : x \in \mathbb N \}$ every set of which has the same label.\par

$ $\par

\textbf{Stage s:}  The triple under consideration is $\langle x_s,y_s,z_s \rangle$.  Let $t$ be the most recent previous stage at which $x_t = x_s$.  We examine two cases: $h_{x_s}^s > h_{x_s}^t$ and $h_{x_s}^s = h_{x_s}^t$.\par

First, suppose that $h_{x_s}^s > h_{x_s}^t$.  We interpret this increase as progress toward verifying $P(e)$.  We enumerate elements as needed to ensure that, for each $x_s$-label, $k$, and $i \leq \langle x_s,k \rangle + h_{x_s}^s$, if $G,G' \in \mathcal S_{x_s}^k$, then $C(G,i) \cap [0,\langle x_s,k \rangle + h_{x_s}^s] = C(G',i) \cap [0,\langle x_s,k \rangle + h_{x_s}^s]$.  If $p_{x_s}^k(i) \leq \langle x_s,k \rangle + h_{x_s}^s$, we pick a member of $C(-1)$ greater than $\langle x_s,k \rangle + h_{x_s}^s$ and every number used in the construction so far, and set $p_{x_s}^k(i)$ equal to the chosen number.  We enumerate $p_{x_s}^k(i)$ into every member of $\mathcal S_{x_s}^k \setminus \{ S_{x_s}^k(i) \}$.  Thus, for each set with $x_s$-label $k$, there is a particular natural number the set does not contain, but which is contained in all other sets with $x_s$-label $k$.\par

Finally, we pick the set of least index in $\mathcal G$ that has not yet been assigned an $x_s$-label and assign it a unique, and previously unused, $x_s$-label.

In the second case, suppose that $h_{x_s}^s = h_{x_s}^t$.  This stability suggests that the outcome will be $\neg P(e)$.  For each $k$ currently in use as an $x_s$-label, we create a new set, $S_{x_s}^k(n_{x_s}^k+1) \in \mathcal S_{x_s}^k$, such that $C(S_{x_s}^k(n_{x_s}^k+1),j) \cap [0,\langle x_s,k \rangle + h_{x_s}^s] = C(S_{x_s}^k(n_{x_s}^k),j) \cap [0,\langle x_s,k \rangle + h_{x_s}^s]$, for $j \leq k + h_{x_s}^s$.  We enumerate $P_{x_s}^k$ into $S_{x_s}^k(n_{x_s}^k+1)$ and set $p_{x_s}^k(n_{x_s}^k+1)$ equal to the least member of $C(-1)$ not used during the construction so far and greater than $\langle x_s,k \rangle + h_{x_s}^s$.  Finally, we enumerate $p_{x_s}^k(n_{x_s}^k+1)$ into every member of $\mathcal S_{x_s}^k \setminus \{S_{x_s}^k(n_{x_s}^k+1)\}$.\par

$ $\par

\textbf{Verification:}  If $P(e)$, then $(\exists x) (\forall y) (\exists z) (R(e,x,y,z))$.  Hence, for some $x$, $h_x^s \rightarrow \infty$.  For infinitely many $s$, $x_s = x$, thus every set in $\mathcal G$ will eventually receive an $x$-label.  At such stages, agreement between sets with the same label is also increased.  Consequently, any two sets in $\mathcal G$ with the same $x$-label are equal.  The family is learned by a machine that searches for the least $x$-label and outputs a code for the first set in $\mathcal G$ that receives the same $x$-label.\par

If $\neg P(e)$, then $(\forall x) (\exists y) (\forall z) (\neg R(e,x,y,z))$.  Fix any machine, $M$.  If $M(\sigma) = \ ?$ for every string $\sigma$ with content contained in a member of $\mathcal G$, $M$ has failed to learn $\mathcal G$ and we are done.  Otherwise, we may pick a string, $\sigma$, such that

\begin{itemize}
\item $M(\sigma) \neq \ ?$;
\item for all $\tau \prec \sigma$, $M(\tau) = \ ?$;
\item for some $G \in \mathcal G$, content$(\sigma) \subseteq G$.
\end{itemize}

Let $k$ be an $x$-label with which $G$ is marked such that ${\rm max}(\mbox{content}(\sigma)) < \langle x_s,k \rangle + h_{x_s}^s$.  Pick a stage, $s$, at which $h_y^s = h_y$ for all $y$ such that $G$ contains a $y$-label less than or equal to $\langle x,k \rangle + h_x$.  At a subsequent stage, $t$, $x_t = x$ and a new set, $G'$, is created containing content$(\sigma)$.  Labels contained in $G'$ are either $y$-labels, $k$, such that $h_y^s$ will never increase or labels enumerated into $G'$ after stage $t$.  The latter are only shared with sets created at subsequent stages.  As a consequence, there is a member of $P_x^k$ never enumerated into $G'$, but contained in $G$.  We conclude that $M$, an arbitrarily chosen learning machine, has failed to TxtFin-learn $\mathcal G$.\par

$ $\par

The computable function that maps $e$ to a code for the $u.c.e.$ family $\mathcal G$ constructed above is a reduction of $P$ to FINL.\par
\end{proof}

\section{TxtEx-Learning}\label{exSec}

We now proceed to describe the arithmetic complexity of TxtEx-learning.  The first $\Sigma_4^0$ description of EXL of which we are aware is due to Sanjay Jain.  Here we present a different formula, but one which explicitly illustrates the underlying structure and serves as a model for the $\Sigma_5^0$ description of BCL given in Section \ref{bcSec}.

\begin{Theorem}\label{exUp}  
EXL has a $\Sigma_4^0$ description.
\end{Theorem}

\begin{proof}  By Theorem \ref{blum} we need only consider computable enumerations when analyzing the complexity of EXL.  Further, observe that if there is a machine, $M$, that TxtEx-learns a family, there is a total machine, $\hat{M}$, that TxtEx-learns the same family.  Specifically, define $\hat{M}(\sigma)$ to be $M(\sigma \upto n)$ for the greatest $n$ such that $M(\sigma \upto n)$ converges within $|\sigma|$ computation stages and define $\hat{M}(\sigma) = 0$ if no such initial segment exists.  Suppose $e$ codes a $u.c.e.$ family $\{L_0, L_1, \ldots\}$.\par

We will define a formula which states that there is a learner such that for every enumeration and every set in the family, if the enumeration is total and enumerates the set, then eventually the hypotheses stabilize and a given hypothesis is either correct or the hypotheses have not yet stabilized.  Syntactically, this can be stated as follows:
\begin{equation}
(\exists a) (\forall k, i) \Big( (M_a \mbox{ is total}) \wedge (\phi_k \mbox{ is total}) \wedge (\phi_k \mbox{ enumerates } L_i) \rightarrow \psi(k,a,i) \Big)\label{exDescription}
\end{equation}
where we define $\psi(k,a,i)$ to be
\begin{align*}
(\exists s) (\forall t>s) &\Big(M_a(\phi_k \upto t)=M_a(\phi_k \upto s) \Big)\wedge (\forall n) \Big(W_{M_a(\phi_k \upharpoonright n)} = L_i \\
&\vee (\exists m>n) \big( M_a(\phi_k \upto m) \neq M_a(\phi_k \upto n) \big) \Big).
\end{align*}
The last formula is $\Delta_3^0$.  Thus, \eqref{exDescription} is $\Sigma_4^0$ and characterizes TxtEx-learning because, for any family coded by a number $e$ which satisfies formula \eqref{exDescription}, there is a learner whose hypotheses converge to correct hypotheses on every computable enumeration and if $e$ fails to satisfy formula \eqref{exDescription}, every learner must fail on some hypothesis for some set in the family.\par

We have, therefore, exhibited a $\Sigma_4^0$ description of EXL.\par
\end{proof}








To achieve the desired completeness result, we now prove that an arbitrary $\Sigma_4^0$ predicate can be reduced to EXL.  The proof utilizes the family described in Example \ref{bcNotEx}.  While not TxtEx-learnable, the family is learnable under more liberal descriptions of learning.  It is, in a sense, just barely not TxtEx-learnable.

\begin{Theorem}\label{exlHard}
EXL is $\Sigma_4^0$-hard.
\end{Theorem}

\begin{proof}
 Let COINF be the index set of all codes for $c.e.$ sets which are coinfinite.  Since COINF is $\Pi_3^0$-complete, it suffices to prove that any predicate of the form $(\exists x) (f(e,x)\in \mbox{COINF})$ for a computable function $f$ can be reduced to EXL.  As in Example \ref{bcNotEx} from the preliminary section, let $H_e = \{ e+x : x\leq |W_e|\}$, $L_e = \{ e+x : x\in \omega \}$, and $\mathcal{F} = \{ H_0, L_0, H_1, L_1, \ldots \}$.  Fix a uniformly computable enumeration of $\mathcal{F}$ where $H_e$ is the $2e+1^{st}$ column of $\mathcal{F}$ and $L_e$ is the $2e^{th}$ column.  For notational convenience, we denote the $e^{th}$ column of $\mathcal{F}$ by $F_e$.  We will define a sequence of $u.c.e.$ families, $\mathcal R_{n,e}$, and choose a computable map $g$ so that $g(e,x)$ is a $\Sigma^0_1$ code for the $u.c.e.$ family $\mathcal{G}_{e,x}$ where, for $x\leq e$
$$\mathcal{G}_{e,x} = \bigcup_{n\in [e,x]} \mathcal R_{n,e}.$$

We construct the $u.c.e.$ families $\mathcal R_{n,e}$ simultaneously for $x < n < e$.  The family, $\mathcal R_{n,e}$, will consist of an infinite number of partial enumerations of $F_n$.  How complete the enumerations are will depend on whether $e\in \mbox{COF}$ or $e\in \mbox{COINF}$.\par

\medskip
\textbf{Stage 0:}  Let $\mathcal R_{n,e}$ be the empty set.\par

\textbf{Stage s:}  Suppose that $[i, i+j] \subseteq W_{e,s}$.  In this case, enumerate $F_{n,j}$, a finite partial enumeration of $F_n$, and the least natural number greater than $n/2$ into the $i^{th}$ column of $\mathcal R_{n,e}$.  Denote this last number by $n_0$.  We include $n_0$ in order to guarantee that the set is nonempty and, if it is a partial enumeration of $H_a$ or $L_a$, it will contain $a$.  Also, for all $k\in [i, i+j]$, enumerate all the elements in the $k^{th}$ column into the $i^{th}$ column and vice versa so that all the columns with indices between $i$ and $i+j$ are identical.

\medskip
There are two cases.  First, suppose $e\in \mbox{COF}$, then there are only finitely many distinct sets in $\mathcal R_{n,e}$; cofinitely many columns of $\mathcal R_{n,e}$ will be identical to $F_n$ and the rest will be finite subsets of $F_n$.  Thus $\mathcal{G}_{e,x}$ will consist of $F_n$ for $x\leq n\leq e$ together with some finite subsets of these sets.  When $e\in \mbox{COINF}$, $\mathcal R_{n,e}$ will contain only finite subsets of $F_n$ and so $\mathcal{G}_{e,x}$ will consist of a collection of finite sets, possibly infinitely many.

Based on $g(e,x)$ and given an arbitrary $\Sigma^0_4$ unary predicate $P$, we define a new map, $h$, which witnesses the reduction of $P$ to EXL.  Since $P$ is $\Sigma^0_4$, it is of the form $(\exists y) (Q(x,y))$ where $Q$ is a $\Pi^0_3$ predicate.  Let $r$ be a one-to-one and computable map witnessing the reduction of $Q$ to COINF.  In other words, $P(x) \leftrightarrow (\exists y) (r(x,y)\in \mbox{COINF})$.  Furthermore, we may assume that 
$$P(e) \rightarrow (\forall^{\infty} y) (r(e,y)\in \mbox{COINF})$$
and
$$\neg P(e) \rightarrow (\forall y) (r(e,y)\in \mbox{COF})$$

Let $s$ be such that for fixed $e$, $\{ s(e,y) \}_{y\in \mathbb N}$ is a computable, strictly increasing, subsequence of $\{ r(e,y) \}_{y\in \mathbb N}$.  The existence of such a subsequence is guaranteed by the fact that $r$ is one-to-one, implying that $\{ r(e,y) \}_{y\in \mathbb N}$ is an unbounded sequence.  For convenience, suppose that $s(e,-1)=0$ for all $e\in \mathbb N$.  Let $h$ be a computable function such that, for $e \in \mathbb N$, $h(e)$ is a code for the $u.c.e.$ family
$$\mathcal{H}_e = \bigcup_{y\in \omega} \mathcal{G}_{s(e,y),s(e,y-1)}.$$

If $\neg P(e)$, then $(\forall y) (r(x,y)\in \mbox{COF})$.  For each $n\in \omega$ there is a $y$ such that $s(e,y-1)\leq n\leq s(e,y)$.  $F_n\in \mathcal{G}_{s(e,y+1),s(e,y)}$, hence $\mathcal{F}\subseteq \mathcal{H}_e$.  Recalling that $\mathcal{F}$ is not TxtEx-learnable, we conclude that $\mathcal{H}_e$ is not TxtEx-learnable.\par

If $P(e)$, then $(\forall^{\infty} y) (r(e,y)\in \mbox{COINF})$ and hence $(\forall^{\infty} y) (s(e,y)\in \mbox{COINF})$.  Pick an $n_0$ such that $(\forall n\geq n_0) (s(e,n)\in \mbox{COINF})$.  For all $n\geq n_0$ $\mathcal{G}_{s(e,n+1),s(e,n)}$ will consist entirely of finite sets.  Furthermore, these sets will, by definition, contain no numbers less than $n_0$.  On the other hand, every set in $\mathcal{G}_{s(e,y+1),s(e,y)}$ for $y\leq n_0$ will contain a number less than or equal to $n_0$ or be finite.  Therefore the whole family is learnable as follows.  Let $M_0$ be a computable function which learns the finite family $\bigcup_{y < n_0+1} \mathcal{G}_{s(e,y),s(e,y-1)}$ and let $M_1$ be a computable function which learns the collection of all finite sets - in other words, a function which interprets the input it receives as a string and outputs a code for the content of that string.  Define
$$M(\sigma) = \begin{cases}
       M_0(\sigma) & n_0\in \mbox{content}(\sigma),\\
       M_1(\sigma) & n_0\notin \mbox{content}(\sigma).
\end{cases}$$

If $M$ is fed an enumeration for a set in the family, then either $n_0$ will eventually appear in the text or it will not.  In either case, the learner will eventually settle on a correct code for the set.\par
$ $\par
We have shown how to reduce an arbitrary $\Sigma^0_4$ predicate to EXL and we may conclude that EXL is $\Sigma^0_4$-hard.\par
\end{proof}

\section{TxtBC-Learning}\label{bcSec}

To prove the upper bound for BCL, we require a result of interest in its own right -- independent of the arithmetic complexity of BCL.

\begin{Theorem}\label{bcEnumBound}
Suppose that $\mathcal G$ is a $u.c.e.$ family.  Either $\mathcal G$ is TxtBC-learnable or, for each computable learner $M$, there is a $\Delta_2^0$ enumeration of a set in $\mathcal G$ that $M$ fails to TxtBC-identify.
\end{Theorem}

\begin{proof}
Fix a $u.c.e.$ family $\mathcal G = \{ G_0, G_1, \dots \}$.  We must prove the following disjunction.  Either:

\begin{enumerate}
\item there is a computable machine which TxtBC-learns $\mathcal G$ or
\item for any computable machine, $M$, either 
\begin{enumerate}
\item there is a $\Delta_2^0$ enumeration for a set $G \in \mathcal G$ on which $M$ stabilizes to an incorrect answer or
\item there is a $\Delta_2^0$ enumeration for a set $G \in \mathcal G$ on which $M$ never stabilizes to codes for a single set.
\end{enumerate}
\end{enumerate}

Assume that statement (2) is false.  We may then fix a learner, $M$, that fails to satisfy statements (2)(a) and (2)(b).  We shall demonstrate that, under this assumption, $\mathcal G$ is TxtBC-learnable by some machine, i.e.~statement (1) is true.  To accomplish this, we perform a construction starting from a computable enumeration, $g(0), g(1), \ldots$, of $G \in \mathcal G$ uniformly obtained from an enumeration of the family.  The construction will follow a strategy designed to produce a $\Delta_2^0$ enumeration witnessing statement (2)(b).  Our assumption that these constructions fail will ultimately yield a method we shall use to build a learner for the family.\par

We construct a $\Delta_2^0$ enumeration, $f$, in stages.  After stage $s$ has completed, the state of the enumeration is a finite partial function, $f_s$.  Let $k_s(0), \ldots, k_s(s)$ denote an increasing reordering of $g(0), \ldots, g(s)$.  In addition, we define a restraint function, $r_s(i)$, and a counter, $i_s$, which monitor the length of the enumeration and the initial segments of $f_s$ on which $M$ exhibits key behavior.  Specifically, $i_s$ counts the number of times $M$ appears to have output hypotheses coding distinct sets on $f_s$, and $r_s(i_s)$ is the length of $f_s$.  For $1\leq j < i_s$, define the hypothesis $h_j^s=M(f_s \upto r_s(j))$ and pick a least witness $x_j^s \in W_{h_j^s,s}\triangle W_{h_{j-1}^s,s}$.  We will call $h_j^s$ and $x_j^s$ the $j^{th}$ hypothesis and witness chosen at stage $s$, respectively.\par

\medskip

\textbf{Stage s+1:}  Let $f_s$, $r_s$, $i_s$, $k_s$, $x_{0}^s, \ldots, x_{i_s}^s$ and $h_{0}^s, \ldots, h_{i_s}^s$ be as obtained from stage $s$.  We shall refer to the preceding collectively as the variables.  Let the finite sequence $k_{s+1}(0),\ldots,k_{s+1}(s+1)$ be an increasing reordering of $g(0), \ldots, g(s+1)$.  Define a set of strings 
$$S(s+1) = \{ \alpha : (|\alpha|,y<s) \wedge(\mbox{content}(\alpha) \subseteq \{ k_{s+1}(0),\ldots,k_{s+1}(s+1) \})\}.$$  
We must consider two possible types of injury at the beginning of the stage.\par

First, suppose that $k_{s+1} \upto (s+1) \neq k_{s} \upto (s+1)$.  Let $j \leq s$ be the least number such that $k_{s+1}(j)\neq k_s(j)$.  Reset the variables to their states at the beginning of stage $j$ (for example, define $f_{s+1}$ to be $f_j$).\par

The second type of injury occurs when a witness is found either to be ``wrong" or ``not least".  We call a witness, $x_j^s$, ``wrong" if $x_j^s \not\in W_{h_j^s,s+1}\triangle W_{h_{j-1}^s,s+1}$ and ``not least" if the tuple $\langle x_j^s, f_s\upto r_s(j) \rangle$ is not the least member of the set
$$\{ \langle y, \alpha \rangle : (f_s\upto r_s(j-1) \prec \alpha) \wedge (y\in W_{h_j^s,s+1}\triangle W_{M(\alpha),s+1}) \wedge (\alpha \in S(s+1))\},$$
where the set is ordered lexicographically and $\alpha <_{llex} \beta$ if $|\alpha| < |\beta|$ or $|\alpha| = |\beta|$ and $\alpha$ is lexicographically less than $\beta$.  Let $j \in \mathbb N$ be least such that $x_j^s$ is either ``wrong" or ``not least" and make the following changes to the variables.

\begin{enumerate}
\item $i_{s+1} = j$.
\item $r_{s+1}(m) = r_s(m)$ for $m<j$ and undefined for $m\geq j$.
\item $f_{s+1} \upto r_s(j-1) = f_s \upto r_s(j-1)$, and $f_{s+1}(x)$ is undefined for $x \geq r_s(j-1)$.
\item Discard $x_{j}^s, \ldots, x_{i_s}^s$ and $h_{j}^s, \ldots, h_{i_s}^s$.
\end{enumerate}

Having dealt with all required injury, we proceed to the actions of the stage.  In particular, we search for the $<_{llex}$-least pair in the set
$$\{ \langle y, \alpha \rangle : (f_{s+1} \prec \alpha) \wedge (y\in W_{h_{i_s}^s,s+1}\triangle W_{M(\alpha),s+1}) \wedge (\alpha \in S(s+1))\}.$$

If a least such pair, $\langle y,\alpha \rangle$, is found, update the variables to reflect the successful search for an extension:

\begin{enumerate}
\item Increment $i_{s+1}$.
\item Extend $f_{s+1}$ to $\alpha$.
\item Define $r_{s+1}(i_{s+1})$ to equal the length of $f_{s+1}$.
\item Update $x_{j}^{s+1} = x_j^s$ and $h_{j}^{s+1} = h_{j}^s$ for $j < i_{s+1}$.
\item Define $x_{i_{s+1}}^{s+1} = y$ and $h_{i_{s+1}} = M(\alpha)$.
\end{enumerate}

On the other hand, if no such pair can be found, end the stage with no further changes.\par
\bigskip

Now suppose that $\lim_{s \rightarrow \infty} i_s = \infty$.  Then, for any $n$, the $n^{th}$ hypothesis and witness will be changed at most finitely many times.  Therefore, $\lim_{s \rightarrow \infty} x_n^s$ and $\lim_{s \rightarrow \infty} f_s\upto n$ exist for every $n$, in which case our construction has produced an enumeration of $G$ which is $\Delta_2^0$ and on which the hypothesis stream generated by $M$ includes hypotheses that code different sets infinitely often.  This is, of course, impossible since, by assumption, $M$ TxtBC-learns $\mathcal G$ from $\Delta_2^0$-enumerations.  Therefore, $\lim_{s \rightarrow \infty} i_s \not = \infty$.\par

The construction was performed using a computable and uniformly obtained enumeration $g(0),g(1),\ldots$ of $G$.  The purpose of using a computable enumeration was to ensure that $\lim_{s\rightarrow \infty} f_s$ was $\Delta_2^0$.  Because each pair of extension and witness are chosen in a canonical manner that is independent of the enumeration, any two instances of the construction will eventually select the same pair despite using different enumerations of $G$.  This can be proved inductively.  Suppose two different enumerations have produced two finite partial functions that agree on an initial, possibly empty, segment.  Take the first point of disagreement.  The choices of extension made at the point when the functions disagree cannot both be $<_{llex}$-least, therefore one will change at a subsequent stage.\par

Define a computable function $\psi$ such that $\psi(\sigma,s) = \tau$, where $\tau$ is the partial function that results from performing the construction on an initial segment, $\sigma$, of an enumeration after $s$ stages of computation.  Let $\hat{M}(\sigma) = M(\psi(\sigma,|\sigma|))$.  Fix an arbitrary enumeration $q(0), q(1), \ldots$ of $G$. Since $M$ TxtBC-learns $\mathcal G$ from $\Delta_2^0$-enumerations, there must be a longest partial function, $\alpha$, that is cofinitely often extended by $\psi(q\upto s,s)$.  For any $\beta$ such that content$(\beta) \subseteq G$, we have $W_{\hat{M}(\alpha)} = W_{\hat{M}(\alpha \hat{\ }\beta)} = G$.  Because $G$ is an arbitrary member of $\mathcal G$, $\hat{M}$ succeeds in TxtBC-learning $\mathcal G$.\par

Since we have proved that $\mathcal G$ is TxtBC-learnable assuming only that $\mathcal G$ is TxtBC-learnable from $\Delta_2^0$-enumerations, we have proved the desired claim.

\end{proof}

The above result allows us to place a bound on the complexity of the enumerations that must be considered when searching for an enumeration that witnesses a failure of TxtBC-learning.  The next result applies Theorem \ref{bcEnumBound} to obtain an upper bound on the complexity of BCL -- the first half of the completeness result for BCL.

\begin{Theorem}
BCL has a $\Sigma_5^0$ description.
\end{Theorem}

\begin{proof}
Let $M$ be an arbitrary learner and $i$ an index for a set in the family $\mathcal F = \{ F_0, F_1, \ldots \}$.  From a computable function, $f$, define a sequence of functions, $\{ f_s \}_{s \in \mathbb N}$, by $f_s(x) = f(s,x)$.  Let $\phi(M,f,i)$ be the formula
\begin{align*}
(\forall n,s) (&W_{M(f_s\upharpoonright n)} = F_i \vee (\exists n'>n) (\exists s'>s) (W_{M(f_{s'}\upharpoonright n')} \not = W_{M(f_s \upharpoonright n)} \\
&\wedge (\forall s''>s')(f_{s''} \upto n' = f_{s'} \upto n'))).
\end{align*}

In words, $\phi(M,f,i)$ asserts that for any stage, $s$, and initial segment, $f_{s}\upto n$, either the hypothesis $M(f_{s}\upto n)$ is correct or there is a later stage and longer initial segment on which the $\Delta_2^0$-enumeration has stabilized and on which $M$ outputs a code for a different set.\par

Define $\psi(M,f)$ to be the formula
$$(\exists n) (\forall n'>n) (\forall s) (W_{M(f_s \upharpoonright n)} = W_{M(f_s \upharpoonright n')} \vee (\exists s'>s) (f_{s'}\upto n' \not = f_s \upto n'))$$
and $\xi(f,i)$ to be
\begin{align*}
&(\forall n) (\exists s) (\forall t>s) (f_s(n) = f_t(n)) \\
\wedge &(\forall n,s) (\exists u,t>s) ((f_s(n) \in F_{i,u}) \vee (f_s(n) \neq f_t (n))) \\
\wedge &(\forall x,u) (\exists n,s) (\forall t>s) (x \in F_{i,u} \rightarrow f_s(n) = f_t(n) \wedge f_s(n) = x).
\end{align*}

If $\psi(M,f)$, then there is an initial segment of length $n$ such that for any stage, $s$, and greater length, $n'$, there are two possibilities.  One, the hypotheses $M$ outputs on $f_s \upto n$ and $f_s \upto n'$ code the same set.  Two, there is a subsequent stage, $s'$, at which the $n'$ length initial segment changes: $f_{s'}\upto n' \not = f_s \upto n'$.  The formula $\xi(f,i)$ asserts that $f_s$ converges to an enumeration of $F_i$ as $s$ goes to infinity.  In particular, $\lim_{s\rightarrow \infty} f_s(n)$ exists for all $n$ and $\lim_{s\rightarrow \infty} f_s(n) = x$ if and only if $x \in F_i$.

We must prove that the following is $\Sigma_5^0$ and equivalent to $e\in$ BCL, where $e$ codes a $u.c.e.$ family $\{F_0, F_1, \ldots \}$.
\begin{align}
(\exists M) (\forall f,i) (\xi(f,i) \rightarrow \psi(M,f) \wedge \phi(M,f,i)).\label{formula}
\end{align}

Observe that $\psi(M,f)$ is $\Sigma_3^0$.  The formula $\phi(M,f,i)$ universally quantifies over the disjunction of a $\Pi_2^0$ formula and a $\Sigma_2^0$ formula.  Thus, $\phi(M,f,i)$ is $\Pi_3^0$.  Since $\xi(f,i)$ is the conjunction of three $\Pi_3^0$ formulas, $\xi(f,i)$ is $\Pi_3^0$.  From this, we conclude that 
$$\xi(f,i) \rightarrow \psi(M,f) \wedge \phi(M,f,i)$$ 
is $\Delta_4^0$.  Consequently, \eqref{formula} is $\Sigma_5^0$.\par

To complete the proof, we must verify that any family that satisfies \eqref{formula} is TxtBC-learnable.  If $\xi(f,i)$, then $f$ converges to a $\Delta_2^0$ enumeration of $F_i$.  From $\psi(M,f)$, we have that there is an initial segment of the enumeration given by $f$ such that, on longer initial segments of $f$, either the output hypotheses code the same set, or the $\Delta_2^0$ enumeration has not yet stabilized.  Finally, $\phi(M,f,i)$ states that for any initial segment either the hypothesis output by the learner is correct or the learner will output a later hypothesis that is different on an initial segment of $f$ that has stabilized.\par

Thus, if \eqref{formula} is true, there is a computable learning machine $M$ such that for any $\Delta_2^0$ enumeration $f$, $M$ converges to consistent hypotheses on $f$ and, if it has not yet output a correct hypothesis, it will change the content of its hypothesis at a later stage.  This is clearly equivalent to TxtBC-learning $\mathcal F$ from $\Delta_2^0$-enumerations.  By Theorem \ref{bcEnumBound}, TxtBC-learning from $\Delta_2^0$-enumerations is equivalent to TxtBC-learning from arbitrary enumerations for $u.c.e.$ families.\par
\end{proof}

We present the lower bound in a modular fashion.  The construction describes an attempt to diagonalize against every possible learner, which succeeds only if a given $\Sigma_5^0$ predicate is false.  A single step of the diagonalization is proved as a lemma.

\begin{Lemma}\label{bcLemma}
Let $M = M_m$ be a computable learning machine and $W_e$ a $c.e.$ set.  There is a family $\mathcal F_{m,e}$, uniformly computable in $m$ and $e$, such that:
\begin{enumerate}
\item If $W_e$ is coinfinite, then $\mathcal F_{m,e}$ is not TxtBC-learnable by $M$, but the family is TxtBC-learnable.
\item If $W_e$ is cofinite, then $\mathcal F_{m,e}$ is uniformly TxtBC-learnable in both $m$ and $e$.

\end{enumerate}
\end{Lemma}

\begin{proof}
The construction will be performed in stages.  During the stages, steps of a diagonalization process will be attempted, although these steps may not be completed.  The diagonalization is against the learner $M$ and a step of the diagonalization will be complete when a string is found on which the learner outputs, as a hypothesis, a code for a set that includes an element not in the content of the enumeration it has been fed.  Such an element will be called a speculation.  To be explicit, we define a natural number, $x$, to be a \textit{speculation} of $M$ on input $\sigma$ if for some $s\in \mathbb N$, $x\in W_{M(\sigma),s}$ and $x\notin \mbox{content}(\sigma)$. \par 

We will build a family $\mathcal F_{m,e} = \{A, B_0, B_1, B_2, \ldots\}$.  At each step $i$, the set $B_i$ is initialized with the contents of the set $A$.  A set, $C$, of speculations will be maintained.  We reserve the $0^{th}$ and $1^{st}$ columns of each set for markers.  If $\langle 0,j \rangle \in B_i$, then $B_i$ is said to have been tagged with $j$.  Every set in the construction will contain $\langle 1,\langle 0,m\rangle\rangle$ and $\langle 1,\langle 1,e\rangle\rangle$ where $m$ is a code for $M$.  The rest of the construction occurs off the $0^{th}$ and $1^{st}$-columns and the $0^{th}$-column of $A$ is left empty.  We now proceed with the construction of $\mathcal F_{m,e}$.\par
\medskip

Fix $M$ and $W_e$.\par
\textbf{Stage 0:}  $C, A, B_0, B_1, \ldots$ are all empty.  Enumerate $0$ into $A$.  Set $\sigma_0 = 0$.\par

\textbf{Stage s:}  Suppose the first $i$ steps have been completed.  By $C, A, B_0, \ldots , B_{i+1}$ we mean those sets in their current state.  We are thus in the midst of step $(i+1)$.   Let $w_0, w_1, \ldots , w_i$ enumerate the current members of $C$, where the index reflects the order in which they were chosen.  Enumerate into each of the sets $A, B_0, \ldots , B_{i+1}$ those $w_j$ having $j\in W_{e,s}$.\par

Next, we search for the least speculation, $x\leq s$, of $M$ on input $\sigma_s \hat{\ } \alpha$, for some $\alpha$ with $|\alpha|\leq s$, $\max(\mbox{content}(\alpha))\leq s$, and $\mbox{content}(\alpha) \cap (C\setminus A) = \emptyset$.  If no speculation is found, pick the least number neither in $C\setminus A$ nor the marker columns and enumerate this number into $B_{i+1}$, after which we end the current stage of the construction.  If a speculation, $x$ witnessed by a string $\alpha$, is found, enumerate $x$ into $C$ and enumerate the members of $\{y: y\notin (C\setminus A) \wedge y\leq \mbox{max(content}(\alpha))\}$ into $A$.  Enumerate $\langle 0,i+1\rangle$ into $B_{i+1}$.  From this point on, only $W_e$ is allowed to enumerate anything further into $B_{i+1}$.  Step $i+2$ is now initiated by enumerating every element of $A$ into $B_{i+2}$.  Finally, we set $\sigma_{s+1} = \sigma_s \hat{\ } \alpha \hat{\ } \beta$, where $\beta$ is an increasing enumeration of $\{y: y\notin (C\setminus A) \wedge y\leq \mbox{max(content}(\alpha))\}$.  This ends the current stage of the construction.\par
\medskip

Observe that $C \setminus A$ are the speculations that, at the current stage, have not been enumerated into $A$.\par

For coinfinite $W_e$ there are two possibilities.  If infinitely many steps complete, there is a subsequence $\{\tau_s\}_{n\in \mathbb N}$ of $\{\sigma_s\}_{n\in \mathbb N}$ such that $W_{M(\tau_s)} \neq A$, for each $s\in \mathbb N$.  Since $\sigma_s$ and $\sigma_t$ are compatible for all $s,t \in \mathbb N$, the computable function $f(n) = \sigma_n(n)$ enumerates $A$.  Thus, we have an enumeration for a set in the family on which $M$ fails to converge to the correct set.  If only finitely many steps complete, then there is a string $\sigma$, equal to $\sigma_s$ for some $s\in \mathbb N$, that has no extension witnessing speculation by $M$.  The content of $\sigma$ is contained in the last nonempty $B_i$, and $B_i$ will become a cofinite set.  Since $M$ engages in no speculation beyond $\sigma$, $M$ must only output codes for finite sets, thus on any enumeration of $B_i$ that begins with the string $\sigma$, $M$ fails to TxtBC-learn~$B_i$.  \par

Depending on the outcome of the construction, but independent of $W_e$, we can define a learning machine $N_0$ that succeeds in TxtBC-learning $\mathcal F_{m,e}$.\par

\textit{Case 1:}  Suppose infinitely many steps of the construction complete.  Define $N_0$ to be a learner that outputs a code for $A$ on any input string unless the string contains $\langle 0,i \rangle$ for some $i$, in which case it outputs a code for $B_i$.  $N_0$ succeeds in TxtBC-learning $\mathcal F_{m,e}$.\par

\textit{Case 2:}  If the $j^{th}$-step is the last step initiated, define $N_0$ to be a learner that, on input $\sigma$, simulates the construction for $A$ and outputs a code for one of $A$, $B_0,B_1,\ldots ,B_j$.  If content$(\sigma) \subseteq A$ and $\langle 0,i \rangle \notin \mbox{content}(\sigma)$ for any $i < j$, $N_0(\sigma)$ codes $A$.  If $\langle 0,i\rangle\in \mbox{content}(\sigma)$, $N_0(\sigma)$ is a code for $B_i$.  Otherwise, $N_0(\sigma)$ is a code for $B_j$ and $N_0$ has TxtBC-learned $\mathcal F_{m,e}$.\par
$ $\par

Next, we define a machine that can learn $\bigcup_{e\in \mbox{COF}, m\in \mathbb N} \mathcal F_{m,e}$.  Fix $e\in \mbox{COF}$. To distinguish it from the completed set, let $A_s$ denote a simulation of the construction of $A$ at stage $s$.
$$N(\sigma) = \begin{cases}
       0  &\mbox{if }  \mbox{card}((\{1\}\oplus \mathbb N)\cap \mbox{content}(\sigma))\leq 1, \\
       N_{m,e}(\sigma) & \mbox{if } \langle 1,\langle0,m\rangle \rangle,\langle 1,\langle1,e\rangle \rangle\in \mbox{content}(\sigma).
\end{cases}$$

The $N_{m,e}$ will be defined below.  Each $N_{m,e}$ need only TxtBC-learn the family resulting from the construction based on $M$ and $W_e$.  For $i \in \mathbb N$, let $A^*,B_i^*$ and $(A\cup B_i \setminus (\{0\} \oplus \mathbb N))^*$ denote $\Sigma^0_1$-codes for $A, B_i$ and $A\cup B_i \setminus (\{0\} \oplus \mathbb N)$, respectively.  These codes can be computably derived from $m$ and $e$.  We define $N_{m,e}$ as follows:
$$N_{m,e}(\sigma) = \begin{cases}
	B_i^* & \mbox{if } \langle 0, i\rangle \in \mbox{content}(\sigma), \\
	A^* & \mbox{if } \langle 0, i\rangle \notin \mbox{content}(\sigma) \wedge \mbox{content}(\sigma)\subseteq A_{|\sigma|}, \\
	(A\cup B_k \setminus (\{0\} \oplus \mathbb N))^*  & \mbox{otherwise,}
\end{cases}$$
where $k$ denotes the greatest index of a set that has been used in the simulated construction up to stage $|\sigma|$.\par

To determine if $N_{m,e}$ TxtBC-learns the family, we must consider four cases, depending on the outcome of the construction and which set, $D \in \mathcal F_{m,e}$, is enumerated to $N_{m,e}$.\par

\textit{Case 1:}  Suppose $D$ has a tag on the $0^{th}$-column; in other words, there exists $i \in \mathbb N$ such that $D = B_i$.  If the construction completes $l < \infty$ steps, then $i < l$.  Otherwise, $D$ may be any of the $B_i$.  Once $\langle 0, i \rangle$ has appeared in the enumeration, the learner will hypothesize $B_i^*$ and never change hypothesis.\par

\textit{Case 2:}  Suppose $D = B_j$ where $j$ is the index of the final, but incomplete, step of the construction.  Since no $\langle 0, i\rangle$ will ever be enumerated into $B_j$, the first case of $N_{m,e}$ will never be satisfied.  Cofinitely, since $A$ is a finite set and $B_j$ is not, the second case will not be satisfied either.  ($A$ is finite because only finitely many steps of the construction complete.)  Thus, cofinitely, the learner will output $(A\cup B_k \setminus (0 \oplus \mathbb N))^*$, where $k$ is updated at each stage to reflect the most recent addition to the family during the construction.  Eventually $k$ will stabilize to $j$, after which time the learner's hypotheses will always be correct.\par

\textit{Case 3:}  Suppose $D = A$, where only finitely many steps complete.  Since $A$ is finite, for all but finitely many $s$, $A_s = A$ and we may replace $A_{|\sigma|}$ with $A$ in the second case of the definition of $N_{m,e}$.  Since no tag will ever be enumerated into the $0^{th}$-column of $A$, the first case will never be satisfied and eventually the second case will always be satisfied and $N_{m,e}$ outputs $A^*$ cofinitely.\par

\textit{Case 4:}  Finally, suppose $D = A$, where infinitely many steps complete.  Note that because the learner is receiving an arbitrary enumeration, there need not be any correlation between the enumeration given to the learner and the enumeration of the simulation $A_s$.  It is quite possible that the second case will be true only infinitely often.  The first case, however, is never satisfied.  All that remains is to prove that eventually the third case only produces correct hypotheses.  Since $e\in \mbox{COF}$, we may choose $s$ such that $[s, \infty) \subseteq W_e$.  The set $C \setminus A$ is finite and $(C \setminus A) \cap B_i = \emptyset$.  Thus, $A\cup B_i \setminus (0 \oplus \mathbb N) = A$ for $i \geq s$.\par
\medskip

Thus $N$ succeeds in TxtBC-learning the following, possibly non-$u.c.e.$, family
$$\bigcup_{e\in \mbox{COF}, m\in \mathbb N} \mathcal F_{m,e}.$$
and thus can TxtBC-learn any subfamily.\par
\end{proof}

\begin{Theorem}\label{bcLowBound}
BCL is $\Sigma_5^0$-hard
\end{Theorem}

\begin{proof}
We wish to reduce an arbitrary $\Sigma^0_5$ predicate $P(e)$ to BCL.  For an arbitrary $\Sigma_4^0$ predicate $Q(e)$, there is a $\Sigma_2^0$ predicate, $R(e,x,y)$, such that the following representation can be made:
\begin{align*} 
Q(e) &\leftrightarrow (\exists a) (\forall b) (R(e,a,b))  \\
&\leftrightarrow (\exists \langle a,s\rangle) [((\forall b) (R(e,a,b)))\wedge ((\forall a'<a) (\exists s'\leq s) (\neg R(e,a',s'))) \\
&\indent \wedge ((\exists a'<a) (\forall s'<s) (R(e,a',s')))] \\
&\leftrightarrow (\exists! \langle a,s\rangle) [((\forall b) (R(e,a,b)))\wedge ((\forall a'<a) (\exists s'\leq s) (\neg R(e,a',s'))) \\
&\indent \wedge ((\exists a'<a) (\forall s'<s) (R(e,a',s')))].
\end{align*}
Since the predicate 
$$((\forall b) (R(e,a,b)))\wedge ((\forall a'<a) (\exists s'\leq s) (\neg R(e,a',s')))\wedge ((\exists a'<a) (\forall s'<s) (R(e,a',s')))$$
is $\Pi^0_3$, for a suitable computable function $g$,
$$Q(e) \rightarrow (\exists! x) (g(e,x)\in \mbox{COINF})$$
and
$$\neg Q(e) \rightarrow (\forall x) (g(e,x)\in \mbox{COF}).$$

Applying the above to $P(e)$, the arbitrary $\Sigma_5^0$ predicate under consideration, we may define a computable function $f$ such that

\begin{align*}
&P(e)\rightarrow (\exists x) [ (\forall x'>x) (\forall y) (f(e,x',y)\in \mbox{COF}) \\
&\indent \indent \indent \wedge (\forall x' \leq x) (\exists^{\leq 1} y) (f(e,x',y)\in \mbox{COINF}) ]
\end{align*}
and
$$\neg P(e)\rightarrow (\forall x) [(\exists! y) (f(e,x,y)\in \mbox{COINF}) ].$$

We will now define a family $\mathcal G_e$ from $e$ such that $\mathcal G_e$ will be learnable if and only if $P(e)$.  Define
$$\mathcal G_e = \bigcup_{x,y\in \mathbb N} \mathcal F_{x,f(e,x,y)}.$$

\textit{Case 1:}  Suppose $\neg P(e)$.  Then for every $x$, there is a $y$ for which $f(e,x,y)\in \mbox{COINF}$.  From this we conclude that for each computable learner, $M$ coded by $m$, there is a $y$ such that $f(e,m,y)\in \mbox{COINF}$.  $\mathcal G_e$ contains a subfamily, $\mathcal F_{m,f(e,m,y)}$, that $M$ cannot TxtBC-learn.  Thus, $\mathcal G_e$ is not TxtBC-learnable.\par

\textit{Case 2:}  Suppose $P(e)$ and let $x_0$ be such that $(\forall x\geq x_0) (\forall y) (f(e,x,y)\in \mbox{COF})$.  Let $a_0, a_1,\ldots , a_k$ enumerate the numbers less than $x_0$ such that, for unique corresponding $b_0, b_1,\ldots , b_k$, we have $f(e,a_i,b_i)\in \mbox{COINF}$ and let $K_i$ be a computable machine that learns $\mathcal F_{a_i,f(e,a_,b_i)}$.  The existence of such a machine is guaranteed by Lemma \ref{bcLemma}.  Using the machine $N$ from the proof of Lemma \ref{bcLemma}, define a computable machine $M$ on input string  $\sigma$ by
$$M(\sigma) = \begin{cases}
	K_i(\sigma) & \mbox{if } \langle 1,\langle0,a_i\rangle \rangle,  \langle 1,\langle1,b_i\rangle \rangle\in \mbox{content}(\sigma) \mbox{ for } i\leq k, \\
	N(\sigma) & \mbox{otherwise.}
\end{cases}$$

If an enumeration of a set in the subfamily $\mathcal F_{a_i,f(e,a_i,b_i)}$ is fed to $M$, then eventually a tag in the $1^{st}$-column will appear identifying it as such.  Cofinitely often, the appropriate $K_i$ will be used to learn the enumeration.  If the enumeration is for a set from $\mathcal F_{x,f(e,x,y)}$ with either $x\not = a_i$ or $y\not = b_i$ for any $i\leq k$, then $N$ will be used.  From Lemma \ref{bcLemma}, it is known that $N$ is capable of TxtBC-learning $\mathcal F_{x,f(x,y)}$ for any $x$ provided that $y\in \mbox{COF}$.\par

We conclude that BCL is $\Sigma_5^0$-hard.\par
\end{proof}

\section{TxtEx$^*$-learning}

Our final collection of results borrows from the BCL lower bound arguments as well as the EXL description, given in Section \ref{bcSec} and Section \ref{exSec}, respectively.  We begin with a $\Sigma_5^0$ description of EXL$^*$.

\begin{Theorem}\label{ex*Up}
EXL$^*$ has a $\Sigma_5^0$ description.
\end{Theorem}

\begin{proof}
Suppose $e$ is a code for a $u.c.e.$ family $\{L_0, L_1, \ldots\}$.  By Corollary \ref{blum*}, a family that is TxtEx$^*$-learnable from computable enumerations is TxtEx$^*$-learnable from arbitrary enumerations.  Further, observe that if there is a machine, $M$, that TxtEx$^*$-learns a family, there is a total machine, $\hat{M}$, that TxtEx$^*$-learns the same family.  Specifically, define $\hat{M}(\sigma)$ to be $M(\sigma \upto n)$ for the greatest $n$ such that $M(\sigma \upto n)$ converges within $|\sigma|$ computation stages and define $\hat{M}(\sigma) = 0$ if no such initial segment exists.  We proceeed with a formula nearly identical to the description of TxtEx-learning.  Let $D_0, D_1, \ldots$ be a canonical, computable enumeration of the finite sets.  Consider the formula
\begin{align}
(\exists a) (\forall k, i) (\exists \ell) \Big( (M_a \mbox{ is total}) \wedge (\phi_k \mbox{ is total}) \wedge (\phi_k \mbox{ enumerates } L_i) \rightarrow \psi(k,a,i,\ell) \Big)\label{ex*description}
\end{align}
where we define $\psi(k,a,i,\ell)$ to be
\begin{align*}
(\exists s) (\forall t>s) &\Big( M_a(\phi_k \upto t)=M_a(\phi_k \upto s) \Big)\wedge (\forall n) \Big( W_{M_a(\phi_k \upharpoonright n)} \triangle L_i = D_\ell \\
&\vee (\exists m>n) \big( M_a(\phi_k \upto m) \neq M_a(\phi_k \upto n)\big) \Big).
\end{align*}

The only difference between the above formula and that of Theorem \ref{exUp} is an additional existential quantifier over finite sets.  Just as before, if a family satisfies formula \eqref{ex*description} then there is a computable machine that identifies every computable enumeration for a set in the family.\par

\end{proof}

\begin{Lemma}\label{ex*Lemma}
Let $M = M_m$ be a computable learning machine and $W_e$ a $c.e.$ set.  There is a family $\mathcal F_{m,e}$, uniformly computable in $m$ and $e$, such that:
\begin{enumerate}
\item If $W_e$ is coinfinite, then $\mathcal F_{m,e}$ is not TxtEx$^*$-learnable by $M$, but the family is TxtEx$^*$-learnable.
\item If $W_e$ is cofinite, then $\mathcal F_{m,e}$ is uniformly TxtEx$^*$-learnable in both~$M$ and~$e$.

\end{enumerate}
\end{Lemma}

\begin{proof}
Fix a machine $M = M_m$ and a $c.e.$ set $W_e$.  We will construct a family, $\mathcal F_{m,e} =\{A,L_1, R_1, L_2, R_2, \ldots \}$ in stages.  Each set in $\mathcal F_{m,e}$ will have two columns ($\{ \langle 0,x \rangle : x\in \mathbb N \}$ and $\{ \langle 1,x \rangle : x\in \mathbb N \}$) reserved for markers.  Every set in $\mathcal F_{m,e}$ contains $\langle 0,\langle m,0 \rangle \rangle$ and $\langle 0,\langle e,1 \rangle \rangle$ and $A$ contains the marker $\langle 0,\langle 0,3 \rangle \rangle$ as well.  Unless otherwise indicated, any action during the construction is performed on the complement of the reserved columns.  We identify this complement with $\mathbb N$ as it is a computable copy.  At any stage of the construction, at most one pair of sets, $L_n$ and $R_n$, will be actively involved in the construction.  When there is such a pair, we call it the active pair and maintain an associated function, $r_n$, which stores information about that pair.\par

\medskip

\noindent \textbf{Stage 0:}  Search for the least string, $\sigma$, on which $M$ outputs a hypothesis.  Set $\sigma_0$ equal to $\sigma$ and enumerate content$(\sigma_0)$ into $A$.\par

\medskip

\noindent \textbf{Stage s+1:}  Let $\sigma_0 \prec \ldots \prec \sigma_s$ be the sequence of strings passed to the current stage from stage $s$.  If there is a currently active pair, $L_n$ and $R_n$, then for $j\in W_{e,s+1}$, we enumerate $r_n(j)$ and $r_n(j) + 1$ into both $L_n$ and $R_n$.  We then consider four cases depending on the status of two parameters.  First, the existence of an active pair of sets.  Second, the availability, within computational bounds, of an extension, $\alpha$, of $\sigma_s$ on which $M$ outputs a hypothesis different from its most recent hypothesis.  We only consider the finite set of strings $S = \{ \sigma_s \hat{\ } \tau : (|\tau|<s+1) \wedge (\mbox{content}(\tau) \subset s+1) \}$ in our search for $\alpha$.\par

\smallskip

\textit{Case 1:}  Suppose there is no active pair, but there is an extension, $\alpha \in S$, of $\sigma_s$ such that $M(\alpha) \neq M(\sigma_s)$.  We pick the least such $\alpha$.  Set $\sigma_{s+1} = \alpha \hat{\ } \beta$, where $\beta$ is an increasing enumeration of $\{ x : x \leq \mbox{max(content}(\alpha)) \}$, and enumerate content$(\sigma_{s+1})$ into $A$.\par

\textit{Case 2:}  Next, consider the case where there is neither an active pair nor an $\alpha \in S$ such that $M(\alpha) \neq M(\sigma_s)$.  Let $n\in \mathbb N$ be least such that $L_n$ and $R_n$ have not yet been used in the construction and set $L_n$ and $R_n$ to be the active pair.  Set $\sigma_{s+1} = \sigma_s$ and enumerate content$(\sigma_{s+1})$ into both $L_n$ and $R_n$.  Pick the least even number, $k$, such that $k$ and $k+1$ have not appeared in the construction so far.  Enumerate $k$ into $L_n$, $k+1$ into $R_n$ and set $r_n(0) = k$.\par

\textit{Case 3:}  Let $L_n$ and $R_n$ be the active pair of sets, and suppose $\alpha \in S$ is least such that $M(\alpha) \neq M(\sigma_s)$.  Set $\sigma_{s+1} = \alpha \hat{\ } \beta$, where $\beta$ is an increasing enumeration of $\{ x : x \leq \mbox{max(content}(\alpha)) \}$, and enumerate content$(\sigma_{s+1})$ into $A$.  Next, we ``cancel" the active pair.  Specifically, we enumerate the marker elements $\langle 1,\langle n,0 \rangle \rangle$ and $\langle 1,\langle n,1 \rangle \rangle$ into $L_n$ and $R_n$, respectively, and mark the pair as inactive.\par

\textit{Case 4:}  Finally, assume there is a currently active pair, $L_n$ and $R_n$, but no $\alpha \in S$ such that $M(\alpha) \neq M(\sigma_s)$.  Pick the least even number $k$ larger than any number used in the construction so far, enumerate $k$ into $L_n$, $k+1$ into $R_n$ and set $r_n(i+1)=k$, where $i$ is the greatest value for which $r_n(i)$ is defined.\par

\bigskip

To verify that the above construction produces a family with the desired properties, we must verify three statements:

\begin{enumerate}
\item $\mathcal F_{m,e}$ is TxtEx$^*$-learnable for all $M$ and $e$.
\item If $W_e$ is coinfinite, then $M$ does not TxtEx$^*$-learn $\mathcal F_{m,e}$.
\item If $W_e$ is cofinite, then there is a machine, computable from $m$ and $e$, that TxtEx$^*$-learns $\mathcal F_{m,e}$.
\end{enumerate}

If there is a pair of sets that remains active cofinitely, then $\mathcal F_{m,e}$ is a finite family.  If no such pair exists, then every finite set has a unique marker by which it can be identified and the only infinite set is $A$.  In either case, the family is learnable and we conclude that the first statement is true.\par

To prove the second statement, we must again consider two cases.  Suppose $W_e$ is coinfinite.  If a pair of sets remains active cofinitely, then $M$ outputs the same hypothesis on all extensions of a finite partial enumeration whose content is contained in both members of the pair.  Thus, there are two enumerations, one for each of $L_n$ and $R_n$, on which $M$ converges to the same hypothesis.  The symmetric difference or $L_n$ and $R_n$, however, is infinite as one is co-odd and the other co-even.  If no pair remains active infinitely, then there must be an infinite number of stages during the construction at which $\sigma_0 \prec \sigma_1 \prec \ldots$ are found such that $M(\sigma_s) \neq M(\sigma_{s+1})$ and $A$ is enumerated by $f(n)=\sigma_n(n)$.  In either case, $M$ fails to TxtEx$^*$-learn $\mathcal F_{m,e}$.\par

Finally, we must exhibit a machine that can TxtEx$^*$-learn all possible families $\mathcal F_{m,e}$ where~$e \in \mbox{COF}$.  In particular, a machine that can TxtEx$^*$-learn the following possibly non-$u.c.e.$ family as well as every subfamily:
$$\mathcal G = \bigcup_{e\in \mbox{COF}, m \in \mathbb N} \mathcal F_{m,e}\hspace{.1em}.$$

Since $W_e$ is cofinite for all the families under consideration, observe that $\mathcal F_{m,e}$ consists of a (possibly infinite) number of finite sets and either one or two sets ($A$ or a pair $L_n$ and $R_n$) that are cofinite in the complement of the marker columns.  Fix codes $a_0$, $a_1$ and $a_{m,e}$ such that $W_{a_0} = \emptyset$, $W_{a_1} = \mathbb N \setminus \{ \langle x,y \rangle : (x=0 \vee x=1) \wedge y \in \mathbb N \}$ and $W_{a_{m,e}}$ is the set $A \in \mathcal F_{m,e}$.  For notational ease, let $C(k) = \{ \langle k,x \rangle : x \in \mathbb N\}$.  Define $N_{m,e}$ by
$$N_{m,e}(\sigma) = \begin{cases}
	a_0 & \mbox{if } \mbox{content}(\sigma)\cap C(1) \neq \emptyset, \\
	a_{m,e} & \mbox{if } \langle 0,\langle 0,3 \rangle \rangle \in \mbox{content}(\sigma), \\
	a_1  & \mbox{otherwise.}
\end{cases}$$
Further, define a machine $N$ by
$$N(\sigma) = \begin{cases}
	N_{m,e} & \mbox{if } \langle 0,\langle m,0 \rangle \rangle,\langle 0,\langle e,1 \rangle \rangle \in \mbox{content}(\sigma), \\
	0  & \mbox{otherwise}.
\end{cases}$$

To prove that $N$ learns $\mathcal G$, select an arbitrary $D \in \mathcal G$.  Let $e$ and $m$ be the codes such that $D \in \mathcal F_{m,e}$ for cofinite $W_e$.\par

\textit{Case 1:}  Suppose that, during the construction of $\mathcal F_{m,e}$, no pair of sets remains active infinitely.  In this case, every member of $\mathcal F_{m,e}$ is marked, either with a marker in $C(1)$ or with $\langle 0,\langle 0,3 \rangle \rangle$.  Thus, $N$ succeeds in TxtEx$^*$-learning $\mathcal F_{m,e}$.\par

\textit{Case 2:}  Suppose, on the other hand, a pair of sets remains active cofinitely during the construction.  Let $L_n$ and $R_n$ be that unique pair of sets.  If $D = A$, then $\langle 0,\langle 0,3 \rangle \rangle \in D$ and cofinitely often $N_{m,e}$ hypothesizes $a_{m,e}$.  Every other finite set contains a unique marker and is hence TxtEx$^*$-learnable by $N_{m,e}$.  Finally, if $D = L_n$ or $R_n$ then $D =^* \mathbb N \setminus (C(0)\cup C(1))$.  No initial segment of any enumeration of $D$ contains either a marker in $C(1)$ or the marker $\langle 0,\langle 0,3 \rangle \rangle$.  Thus, $N_{m,e}$ again succeeds in TxtEx$^*$-learning the set.\par

\end{proof}

\begin{Theorem}
EXL$^*$ is $\Sigma_5^0$-hard
\end{Theorem}

\begin{proof}
The proof is identical to that of Theorem \ref{bcLowBound} with one exception; in the conclusion, we use Lemma \ref{ex*Lemma} to justify the claim that $N$ TxtEx$^*$-learns $\mathcal F_{x,f(x,y)}$ for any $x$ provided that $y\in \mbox{COF}$.\par

\end{proof}

\section{Conclusion}

In summary, we have proved $\Sigma_3^0$-completeness for FINL, $\Sigma_4^0$-completeness for EXL and $\Sigma_5^0$-completeness for both BCL and EXL$^*$.  Numerous other learning criteria are known to the theory, but their arithmetic complexities remain to be determined.\par

One question stemming from the above work is to ask if there are natural classes other than $u.c.e.$ families for which these complexity questions can be answered.  Any candidate would have to provide a framework within which an upper bound could be placed on complexity.  If more general classes are considered, observe that complexity can only increase, whereas for more restrictive classes complexity can only decrease.

\bibliographystyle{plain}

\end{document}